\documentclass[letterpaper,11pt]{article}

\usepackage{amsmath,amsfonts,amssymb,amsthm,mathrsfs}
\usepackage{cases}
\usepackage{enumerate}
\usepackage[usenames,dvipsnames]{color}
\usepackage{verbatim} 
\usepackage{array} 
\usepackage{graphicx,epstopdf} 
\usepackage{mathtools}

\newcommand{\E}[0]{\mathbb{E}}

\newtheorem{thm}{Theorem}[section]

\newtheorem{prop}[thm]{Proposition}

\newtheorem{conj}[thm]{Conjecture}

\newtheorem{question}[thm]{Question}

\newcommand{\beq}[1]{\begin{equation}\label{#1}}
\newcommand{\enq}[0]{\end{equation}}

\newcommand{\bn}[0]{\bigskip\noindent}
\newcommand{\mn}[0]{\medskip\noindent}
\newcommand{\nin}[0]{\noindent}

\newcommand{\sub}[0]{\subseteq}
\newcommand{\sm}[0]{\setminus}
\renewcommand{\dots}[0]{,\ldots,}

\newcommand{\ov}[0]{\overline}

\newcommand{\A}[0]{{\cal A}}
\newcommand{\B}[0]{{\cal B}}
\newcommand{\cee}[0]{{\cal C}}

\newcommand{\g}[0]{{\cal G}}
\newcommand{\h}[0]{{\cal H}}

\newcommand{\K}[0]{{\cal K}}

\newcommand{\m}[0]{{\cal M}}

\newcommand{\pee}[0]{{\cal P}}
\newcommand{\Q}[0]{{\cal Q}}

\newcommand{\R}[0]{{\cal R}}
\newcommand{\sss}[0]{{\cal S}}
\newcommand{\T}[0]{{\cal T}}

\newcommand{\ra}[0]{\rightarrow}

\newcommand{\mm}[0]{m}

\renewcommand{\int}[0]{{\rm int}}

\newcommand{\aaa}[0]{a}

\newcommand{\llll}[0]{l}
\newcommand{\RRR}[0]{C}

\newcommand{\0}[0]{\emptyset}

\newcommand{\C}[2]{{{#1}\choose{{#2}}}}
\newcommand{\Cc}[0]{\tbinom}
\newcommand{\ga}[0]{\alpha }
\newcommand{\gb}[0]{\beta }
\newcommand{\gc}[0]{\gamma }
\newcommand{\gd}[0]{\delta }
\newcommand{\gD}[0]{\Delta }
\newcommand{\gG}[0]{\Gamma }

\newcommand{\gl}[0]{\lambda }

\newcommand{\go}[0]{\omega}
\newcommand{\gO}[0]{\Omega}

\newcommand{\gz}[0]{\zeta}
\newcommand{\eps}[0]{\varepsilon }
\newcommand{\vt}[0]{\vartheta}

\newcommand{\1}[0]{{\bf 1}}

\newcommand{\comments}[1]{}
\usepackage[usenames]{color}

\newcommand{\pr}[1][]{\mathbb{P}}

\title{Disproof of a 
packing conjecture of \\ Alon and Spencer}

\author{
H\"{u}seyin Acan
\footnote{Department of Mathematics, Rutgers University}
\thanks{Supported by National Science Foundation Fellowship (Award No.~1502650).}
\\
{\small \texttt{huseyin.acan@rutgers.edu}}
\and
Jeff Kahn \footnotemark[1] \thanks{Supported by NSF grant DMS1501962.}\\
{\small \texttt{jkahn@math.rutgers.edu}}
}
\date{}

\begin{document}
\renewcommand{\thefootnote}{\fnsymbol{footnote}}
\maketitle

\begin{abstract}
A 1992 conjecture of Alon and Spencer says, roughly, that the ordinary random graph
$G_{n,1/2}$ typically admits a covering of a constant fraction of its edges by edge-disjoint,
nearly maximum cliques.  We show that this is not the case.  The disproof is based on
some (partial) understanding of a more basic question:
for $k\ll \sqrt{n}$ and $A_1\dots A_t$ chosen uniformly and independently from the $k$-subsets of
$\{1\dots n\}$, what can one say about
\[
\pr(|A_i\cap A_j|\leq 1 ~\forall i\neq j)?
\]
Our main concern is trying to understand how closely the answers to this and a related question about
matchings
follow heuristics gotten by pretending that certain (dependent) choices are made independently.
\end{abstract}

\section{Introduction}\label{Intro}

Write $G$ for the the random graph $G_{n,1/2}$ and
$f(k)$ ($=f_n(k)$) for the expected number of
$k$-cliques in $G$; that is,
$f(k)=\C{n}{k}2^{-\C{k}{2}}$.
Set
\[
k_0 = k_0(n) = \min \{k:f(k)<1\}
\]
and temporarily (through Conjecture~\ref{ASConj})
set $k=k(n) =k_0-4$.
It is easy to see that $k \sim 2\log_2n$ and that $f(k)$ is at least about
$n^3$ (precisely, $f(k)=\tilde{\gO}(n^3)$, where, as usual, $\tilde{\gO}$ ignores log factors).

We will call a collection of edge-disjoint cliques a \emph{packing}
(and a \emph{t-packing} if it has size $t$).
Write $\nu_k(G)$ for the maximum size of a packing of
$k$-cliques in $G$.  This quantity (with independent sets in place of cliques)
plays a central role in Bollob\'as' celebrated work
\cite{Bollobas} on the chromatic number of $G$, though all he needs from
$\E \nu_k(G)$---the quantity that will interest us here---is the easy
\beq{easylb}
\E \nu_k(G) =\gO(n^2/k^4).
\enq
(His key point is that $\nu_k$ is Lipschitz, so
martingale concentration implies it is (\emph{very}) unlikely to
be significantly smaller than its expectation.)

Of course one always has $\nu_k(G)\leq \C{n}{2}/\C{k}{2}$.
A conjecture of Alon and Spencer, from
the original 1992 edition of \cite{AS} (and subsequent editions),
says that this trivial bound gives the true order of magnitude of
$\E \nu_k(G)$, \emph{viz.}
\begin{conj}\label{ASConj}
$~~~\E \nu_k(G)=\gO(n^2/k^2).$
\end{conj}
\nin
In other words, one can (in expectation) cover a constant fraction of the pairs
from $[n]$ by edge-disjoint $k$-cliques of $G$.
Here we show that this is not correct, even for
somewhat smaller $k$:
\begin{thm}\label{ASWrong}
For each $\RRR$ there is a D so that if
$k=k_0-\RRR$, then
\[
\pr(\nu_k(G) > Dn^2/k^3)< \exp[-n^2/k^2].
\]
\end{thm}

Again, it is easy to see that for $k$ as in Theorem~\ref{ASWrong} we have
\beq{fk}
\tilde{\gO}(n^{C-1}) < f(k) < n^C,
\enq
and that edges of $G$ typically lie in many (at least about $n^{C-3}$) $k$-cliques,
which might suggest plausibility of Conjecture~\ref{ASConj}.
But as we will see below (following Theorem~\ref{TMP2}),
falsity of the conjecture should not be surprising,
though establishing this intuition so far seems less straightforward than one might expect.

\mn

We also observe a slight improvement in the lower bound of \eqref{easylb},
an easy consequence of a seminal result of Ajtai, Koml\'os and Szemer\'edi \cite{aks,sidon}:
\begin{prop}\label{Plb}
$\E \nu_k(G)=\gO((n^2 /k^4)\log k)$.
\end{prop}
\nin
Though it may look like a detail at this point, determining the true order of magnitude of
$\E \nu_k(G)$ still seems to us quite interesting, since it seems to require
understanding more basic issues.
For a guess, we slightly prefer the upper bound, but there are heuristics on both sides.
It is not too hard to see that for a suitable $c$ the \emph{expected} number of
$(cn^2/k^3$)-packings of $k$-cliques is large.

\mn

As above, a collection of sets is a \emph{packing}---or is
\emph{nearly-disjoint}; we will find it convenient to have both terms---if no
two of its members have more than one point in common,
and a $t$-\emph{packing} is a packing of size $t$.
As usual a \emph{matching} is a collection of pairwise \emph{disjoint} sets
and an $m$-matching is a matching of size $m$.
Given $n$ and $k$ (for most of our discussion $k$ need not be
as above), we write $\K$ for $\C{[n]}{k}$.

\mn

Our real interest in this paper is in the validity of
heuristics based on the idea that certain events are close to independent.
We view the next question in this way and will see a second instance in the discussion
around Theorem~\ref{TMP}.

\begin{question}\label{Q}
For
$A_1\dots A_t$ drawn uniformly and independently
from $\K$, what can be said about
\beq{gz}
\gz=\gz(n,k,t):=\pr(\mbox{$A_1\dots A_t$ form a packing})?
\enq
\end{question}
\nin
Of course what we expect here will depend on the parameters.
We assume throughout that
\beq{ksmall}
1\ll k\ll \sqrt{n}.
\enq
As noted below, the case of fixed $k$ is handled in \cite{LL,Keevash}
(with slight changes to our ``natural" answers,
e.g.\ since $\C{k}{2}\not\sim k^2/2$ when $k$ is fixed).
The upper bound in \eqref{ksmall} makes $\pr(|A_i\cap A_j|\geq 2)$
small, without which the problem seems less natural.
(We actually tend to think of $k=\Theta(\log n)$,
the relevant range for Theorem~\ref{ASWrong}.)

\mn

For $k$ as in \eqref{ksmall} and $A,B$ drawn uniformly and independently from $\K$,
\[
\pr ((|A\cap B|\geq 2)\approx k^4/(2n^2);
\]
so thinking of the events
$\{|A_i\cap A_j|\geq 2\}$ as close to independent suggests
\beq{ideal1}
\gz ~\approx ~\left(1-k^4/(2n^2)\right)^{\C{t}{2}}
~\approx~
\exp\left[-\tfrac{t^2k^4}{4n^2}\right].
\enq
Another, more robust way to arrive at the same guess:  the probability that $m:=t\C{k}{2}$ pairs chosen
\emph{independently} (and uniformly) from $\C{[n]}{2}$ are distinct is
\beq{ideal2}
\mbox{$\prod_{i=1}^{m-1}(1-i/\C{n}{2}) $,}
\enq
which agrees (approximately) with the r.h.s.\ of \eqref{ideal1}, provided
\beq{tsmall}
t\ll n^2/k^2.
\enq
It seems not impossible that these heuristics are close to the truth;
precisely, that for $t$ as in \eqref{tsmall},
\beq{ideal3}
\mbox{$\log(1/\gz) \sim t^2k^4/(4n^2)$}
\enq
(when the distinction matters, we use $\log$ for $\ln$),
while for larger $t$ (where
\eqref{ideal1} and \eqref{ideal2} are not so close)
the asymptotics of
$\log(1/\gz)$ are given by \eqref{ideal2}.

\mn

Here we give upper bounds on $\gz$ that
(i) for $t$ relevant to Theorem~\ref{ASWrong}
support the theorem but
fall somewhat short of \eqref{ideal3}, and (ii) agree
with \eqref{ideal3} for slightly smaller $t$.
We will not have anything to say about lower bounds.

\begin{thm}\label{TMP2}
{\rm (a)}
There is a fixed $\gb>0$ such that if $t = Dn^2/k^3 $, then
\beq{zetabd}
\gz < \left\{\begin{array}{ll}
\exp[-\gb D tk]&\mbox{if $e\geq D =\gO(1)$,}\\
\exp[-\gb (\log D)tk]&\mbox{if $D>e$.}
\end{array}\right.
\enq
{\rm (b)}
If $1\ll t\ll n^2/k^3~$ then $~\gz < \exp[-(1-o(1))t^2k^4/(4n^2)]$.
\end{thm}
\nin
(Note for perspective that for $t$ as in (a) the
bound in (b), which is essentially ideal if we have \eqref{tsmall},
becomes $\exp[-(1-o(1))Dtk/4]$.
We won't bother with the silly case $t=O(1)$---which would require
occasionally replacing $t^2$ by $t(t-1)$---and retain the uninteresting
constant bounds for $t=O(n/k^2)$ only because they require no extra effort.)

\mn

Before continuing we observe that this gives Theorem~\ref{ASWrong}.
We may bound $\pr(\nu_k(G) \geq t)$ by
the expected number of $t$-packings
in $G$ ($=G_{n,1/2}$), which is less than
\beq{Ebd}
\gz \Cc{n}{k}^t2^{-\C{k}{2}t}
= \gz \left[\Cc{n}{k}2^{-\C{k}{2}} \right]^t < \gz \exp[Ct\log n],
\enq
where $\gz \Cc{n}{k}^t$ (crudely) bounds the number of $t$-packings
in $\K$, each of which appears in $G$ with probability
$2^{-\C{k}{2}t}$, and the inequality is given by \eqref{fk}.
Now letting $t=Dn^2/k^3$ with $D$ ($>e$) chosen so that $\gb \log D > C$
(where $\gb$ is as in Theorem~\ref{TMP2} and we
recall $k\sim 2\log_2 n$)
and combining \eqref{Ebd} with the second bound in \eqref{zetabd}
gives
\[
\pr(\nu_k(G)> Dn^2/k^3) < \exp[-n^2/k^2]. \tag*{\qed}
\]

\mn

The argument for Theorem~\ref{TMP2}(a)
(the part needed for Theorem~\ref{ASWrong})
is mainly based on Theorem~\ref{TMP} below,
which we next spend a little time motivating.

To begin, we remind the reader that there is a natural entropy-based approach
to problems ``like" that addressed by Theorem~\ref{TMP2};
this approach was introduced by J. Radhakrishnan \cite{JR}
in his proof of Br\'egman's Theorem~\cite{Bregman} and followed more recently in
(e.g.) the Linial-Luria upper bound on the number of Steiner triple systems \cite{LL}
and its extension to more general designs by Keevash \cite{Keevash}.
In our situation the entropy argument works up to a point, but we don't see
how to push it to a proof of Theorem~\ref{TMP2} (or a disproof of Conjecture~\ref{ASConj})
and will take a different approach.

A first simple (but seemingly crucial) idea is that we should choose our packing in two rounds,
the first round specifying just half, say $B_i$, of each $A_i$.
A \emph{necessary} condition for a packing is then:
\[
\mbox{\emph{for each $x\in [n]$,
$\{B_i: x\in A_i\sm B_i\}$ is a matching.}}
\]
Modulo a certain amount of fiddling, this gets us to the following situation,
in which $l$ will be $k/2$.

We assume $\h$ is a nearly-disjoint $l$-graph ($l$-uniform hypergraph) with $n$ vertices and $t$ edges,
and $\m=\{e_1\dots e_m\}$ is a random (uniform) $m$-subset of $\h$,
and are interested in
\[
\xi =\xi_\h(m) =\pr(\mbox{$\m$ is a matching}).
\]
(When we apply this to Theorem~\ref{TMP2}, $t$ will be as in the theorem
and $m$ will be something like $t\llll /n$.)
%
Setting $c = m\llll^2/n$,
we again have a natural value for $\xi$, namely,
\beq{Mideal}
(1-\llll ^2/n)^{\C{m}{2}}~\approx ~\exp[- cm/2],
\enq
gotten by pretending independence of the events $\{e_i\cap e_j\neq\0\}$,
the natural value of whose probabilities is roughly $1-\llll^2/n$.
The next statement is perhaps our main point.

\begin{thm}\label{TMP}
If $t\gg n/l$ and $c=\min\{ml^2/n,t\llll /n\}$, then
\beq{ProbM}
\mathbb{P}(\mbox{$\m$ is a matching}) <
\left\{\begin{array}{ll}
\exp[-\gO( cm)]&\mbox{if $c\leq e$,}\\
\exp[-\gO( (\log c)m)]&\mbox{if $c>e$.}
\end{array}\right.
\enq
\end{thm}
\nin
(We won't get into $t=O(n/l)$.
Of course the theorem evaporates if $t\leq n/l$, since $\h$ itself can then be a matching.)

\mn

Note that here, unlike in Theorem~\ref{TMP2}, we may think of $\h$ as chosen adversarially,
and should adjust expectations accordingly; in particular the probability in \eqref{ProbM}
can easily be zero, so at best we may hope that \eqref{Mideal} offers some guidance on
\emph{upper} bounds.  It's also true that, as shown by the following example, the probability of a
matching can easily be about $(tl/n)^{-m}$, so even the second part of \eqref{ProbM}
(the one that differs more seriously from \eqref{Mideal})
can't be much improved under the stated
hypotheses.
On the other hand---and more interestingly---it could be that \eqref{Mideal} is about
right (as an upper bound) if, say,
$\log(t l/n)\gg c $ ($=ml^2/n$).

\begin{proof}[Example]
For $t=sn/l$ with $l$ a prime power, let $\g$ consist of $s$ parallel
classes of an affine plane of order $l$, and let $\h$ be the disjoint union of $n/l^2$
copies of $\g$.  Then for $m\ll n/l$ and $e_1\dots e_m$ drawn uniformly and
independently from $\h$,
\[
\pr(\mbox{$\{e_1\dots e_m\}$ is a matching}) > s^{-(1-o(1))m},
\]
as follows from the observation that if $\{e_1\dots e_i\}$ is a matching then the number of
edges disjoint from $e_1\dots e_i$ is at least $n/l-i$ (and exactly this if
$\{e_1\dots e_i\}$ meets all copies of $\g$).
\end{proof}

\mn
\emph{Outline}
Theorems~\ref{TMP2} and \ref{TMP} are proved in Sections~\ref{PTMP2} and
\ref{SecMP}, following a quick large deviation review
(mainly for Theorem~\ref{TMP2}(b))
in Section~\ref{Prelim}.
The proof of Proposition~\ref{Plb} is sketched in Section~\ref{LB}.

\mn
\emph{Usage.}
For asymptotics we use $a\ll b$ and $a=o(b)$ interchangeably.
As is common we pretend all large numbers are integers and always assume
$n$ is large enough to support our arguments.  As mentioned earlier,
$\log $ is $\ln$.

\section{Preliminaries}\label{Prelim}
We will need the following ``Chernoff bounds"
(see e.g.\
\cite[Thm.\ 2.1 and Cor.\ 2.4]{JLR}; we won't need to deal with lower tails).
\begin{thm}
\label{thm:Chernoff}
If $X \sim \mathrm{Bin}(n,p)$ and $\mu = \mathbb{E}[X] = np$, then
\begin{align*}
\Pr(X > \mu + t) &< \exp\left[-t^2/(2(\mu+t/3))\right] ~~\forall t>0, \\
\label{Ch2}
\Pr(X > K\mu ) &<\exp[-K\mu \log (K/e)] ~~\forall K.
\end{align*}
\end{thm}
\nin
(Of course the second bound is only of interest for slightly large $K$.
We won't need to deal with lower tails.)

\mn

Though it could be avoided, the following less usual bit of machinery is nice
and will be convenient for us at one point.
Recall that r.v.'s $\xi_1\dots \xi_n$ are {\em negatively associated}
if $\E fg\leq \E f\E g$ whenever there are disjoint $I,J\sub [n]$
for which $f$ and $g$ are increasing functions
of $\{\xi_i:i\in I\}$ and $\{\xi_i:i\in J\}$ (respectively).
As observed in \cite[Lemma 8.2]{DR}, Chernoff-type bounds
usually apply at the
level of negatively associated r.v.'s; we state only what we need in this direction:

\begin{prop}\label{NAprop}
{\rm (a)} If $A_1\dots A_t$ are drawn uniformly and independently from
$\C{[n]}{k}$, then the degrees $d(j) = |\{i:j\in A_i\}|$ ($j\in [n]$) are
negatively associated, as are any r.v.'s $\xi_1\dots \xi_n$
with $\xi_j$ an increasing function of $d(j)$.

\mn
{\rm (b)}
If $\xi=\sum \xi_i$ with the $\xi_i$'s
negatively associated, then for any $\ga$ and $\gl>0$,
\beq{dr}
\mbox{$\pr(\xi > \ga) ~<~e^{-\gl \ga}\E e^{\gl \xi}
~\leq~ e^{-\gl \ga}\prod\E e^{\gl \xi_i}.$}
\enq
\end{prop}

\nin
For (a) see \cite[Propositions 3.1 and 3.2]{EKRI} (as remarked there,
the statement is probably not news to anyone interested in such things).
The content of \eqref{dr} is the second inequality; the
first, included here just for orientation, is the usual use
of Markov's Inequality in proving Chernoff bounds.


\section{Proof of Theorem~\ref{TMP2}}\label{PTMP2}


\begin{proof}[Proof of Theorem~\ref{TMP2}\rm{(a)} (given Theorem~\ref{TMP})]
%
As mentioned in Section~\ref{Intro}, a crucial first idea
is that we should choose the $A_i$'s in two stages.
For simplicity suppose $k$ is even, say $k=2l$.
For $i\in [t]$, let $A_i=B_i\cup C_i$, with $B_i$ uniform from $\C{[n]}{l}$
and $C_i$ uniform from
$\C{[n]\sm B_i}{l}$
(with the choices for different $i$'s independent), and let
$\h=\{B_1\dots B_t\}$.
Let $\pee=\{\mbox{$A_1\dots A_t$ form a packing}\}$
(the event in \eqref{gz}) and
$\Q=\{\mbox{$B_1\dots B_t$ form a packing}\}$.
Of course $\Q$ is a prerequisite for $\pee$,
so we need only show
\beq{MP3}
\pr(\pee|\Q) < \left\{\begin{array}{ll}
\exp[-\gO( D) tk]&\mbox{if $D\leq e$ (say),}\\
\exp[-\gO (\log D)tk]&\mbox{if $D>e$.}
\end{array}\right.
\enq

From this point we fix a packing $\{B_1\dots B_t\}$
and consider the probability
of $\pee$ given $\{\h=\{B_1\dots B_t\}\}$.
The problem is now more about counting than probability:
we want to bound the number of ways of choosing
$\g:=\{C_1\dots C_t\}$
so that the resulting $A_i$'s form a packing.

We may think of choosing the $C_i$'s by first choosing degrees $d_j:=d_\g(j)$ ($j\in [n]$)
satisfying
\beq{sumd}
\mbox{$\sum d_j = tl$}
\enq
and then sets
\[
S_j:=\{i:j\in C_i\}
\]
satisfying, for each $j\in [n]$,
\[
\mbox{$\{B_i:i\in S_j\}$ is a matching}
\]
(another prerequisite for $\pee$).
Of course only a small fraction of such choices correspond to legitimate $C_i$'s,
but this overcount turns out to be affordable.

Given $d_j$'s the number of choices of $S_j$'s as above
is
$
\prod_{j\in [n]} N(d_j)
$,
where $N(d)=N_\h(d)$ is the number of $d$-matchings in $\h$.

Set $u =tl/n$ (the average of the $d_j$'s).
Since $\sum\{d_j:d_j\leq u/2\} \leq un/2$, \eqref{sumd} implies
\beq{sumdj}
\sum\{d_j:d_j> u/2\} \geq un/2 =tl/2.
\enq
For bounding $N(d)$ when $d\leq u/2$, we use the trivial
$   
N(d)\leq \Cc{t}{d}.
$  
For larger $d$, noting that
$u/2= Dn/(16l^2)$,
we may apply Theorem~\ref{TMP} with $m=d$ and $c \geq D/(16)$ (and $t=t$,
so $tl/n = Dn/(8l^2)\gg D$) to obtain
\beq{larged}
N(d)<e^{-Bd}\Cc{t}{d},
\enq
where
\[
B=\left\{\begin{array}{ll} \gO(D)&\mbox{if $D\leq 16e$,}\\
\gO(\log D)&\mbox{otherwise.}
\end{array}\right.
\]

Thus for a particular set of $d_j$'s the number of ways to choose the $S_j$'s is less than
\beq{bd1}
\mbox{$\exp[-B\sum^*d_j]\cdot \prod \Cc{t}{d_j} < e^{-Btl/2}\prod \Cc{t}{d_j} $,}
\enq
where $\sum^*$ runs over $j$ with $d_j> u/2$
(and we use $u=tl/n$ and \eqref{sumdj}).
For the product we have, again using \eqref{sumd},
\beq{bd2}
\mbox{$\prod \Cc{t}{d_j}  < (et)^{\sum d_j}\prod d_j^{-d_j} \leq (et)^{tl} (tl/n)^{-tl} = (en/l)^{tl}$}
\enq
(since convexity of $x\log x$ implies that,
given \eqref{sumd}, $\prod d_j^{d_j}$ is minimum when $d_j = u$ for all $j$).
The (negligible) number of ways
to choose the $d_j$'s is
\beq{bd3}
\mbox{$\C{tl+n-1}{n-1} < n^n$.}
\enq

On the other hand, the (\emph{total}) number of ways of choosing $C_1\dots C_t$
(again, for given $B_i$'s)
is
\[
\Cc{n-l}{l}^t > l^{-t}(en/l)^{tl}
\]
(since $l\ll \sqrt{n}$, Stirling's formula gives $\C{n-l}{l} \sim (2\pi l)^{-1/2 }(en/l)^l$),
and combining this with \eqref{bd1}-\eqref{bd3}
we find that the probability of $\pee$ (given the specified $B_i$'s) is less than
\[
n^n l^t e^{-Btl/2}(en/l)^{tl}(en/l)^{-tl} = e^{-(1-o(1))Btl/2}. \qedhere
\]
\end{proof}


\mn

\begin{proof}[Proof of Theorem~\ref{TMP2}\rm{(b)}]
Set $t=\eps n^2/k^3$ (so $\eps = o(1)$).
Let $\gd$ be some sufficiently slow $o(1)$ and set $t_0=\gd t$.
(We need $ \gd^2\gg\eps$ and, at \eqref{ksmall},
$\exp[-\gO(\gd^2k)]\ll \gd$.)
Set $\h_i=\{A_1\dots A_i\}$ and $\h=\h_t$, and
write $d_i$ and $d$ for degrees in $\h_i$ and $\h$.

We first need to dispose of some pathological situations in which vertices
with very large degrees meet too many edges of $\h$, to which end we set
$a_0=\gd n/k^2$ and
$
W=\{j\in [n]:d(j)\geq a_0\},
$
and consider the event
\[
\mbox{$\Q = \{\sum_{j\in W}d(j) <\gd t_0k\}$}.
\]

\nin
\emph{Claim 1.}
$
~~\pr (\ov{\Q}) < \exp[-t^2k^4/n^2]
$

\begin{proof}
Theorem~\ref{thm:Chernoff}
applied to $d(j) \sim \textrm{Bin}(t,k/n)$ is easily seen to imply that for
any $a\geq a_0$
(using $\gd \gg \eps$ to say $a_0\gg tk/n$),
\[   
\pr(d(j) \geq a) < (e\eps/\gd)^a
\]   
which for $\xi_j:=d(j)\1_{\{d(j)\geq a_0\}}$ implies
\beq{ea0}
\mbox{$\E e^{\xi_j} ~<~ 1 + \sum_{a\geq a_0} a(e\eps/\gd)^a
~<~ \exp [\eps^{\go(1)}].$}
\enq
Moreover, by Proposition~\ref{NAprop}(a), the $\xi_j$'s are negatively associated,
so part (b) of the proposition gives, for $\xi=\sum\xi_j$,
\[
\pr(\ov{\Q})= \pr (\xi >\gd t_0k) < \exp[-\gd t_0k + n\eps^{\go(1)}] = \exp[-(1-o(1))\gd t_0k],
\]
which, since $\gd^2 \gg \eps$, is less than the bound in Claim 1.
(Note also that $\gd t_0k\gg n\eps^{\go(1)}$ is the same as $\gd^2\eps n/k^2\gg \eps^{\go(1)}$.)

\end{proof}

Now let
$\pee_i = \{\mbox{$\h_i$ is a packing}\}$, $\pee=\pee_t$,
$
W_i=\{j\in [n]:d_i(j)\geq a_0\}
$
and
$
\mbox{$\Q_i= \{\sum_{j\in W_i}d_i(j) <\gd t_0k\}$}.
$

\mn
\emph{Claim 2.}
For $i>t_0$, if $A_1\dots A_{i-1}$ satisfy $\pee_{i-1}\Q_{i-1}$, then
\[
\pr(|A_i\cap A_j|\leq 1 ~\forall j<i) < \exp [-(1-o(1))ik^4/(2n^2)].
\]
Once this is established, we have
(noting that $\pee\Q =\cap (\pee_i\Q_i)$, since in fact $\pee_1\supseteq \cdots \supseteq \pee_t=\pee$
and $\Q_1\supseteq \cdots \supseteq \Q_t=\Q$),
\begin{eqnarray*}
\mbox{$\pr(\pee) $}\leq \mbox{$\pr(\ov{\Q}) +\pr(\pee \Q)$}
&\leq &\mbox{$\pr(\ov{\Q}) + \prod_i\pr(\pee_i\Q_i|\pee_{i-1}\Q_{i-1}),$}\\
&\leq &\mbox{$\pr(\ov{\Q}) + \prod_{i=t_0}^t\pr(\pee_i|\pee_{i-1}\Q_{i-1}),$}
\end{eqnarray*}
which, according to Claims 1 and 2, is less than
\[
\mbox{$\exp[-\frac{t^2k^4}{n^2}] + \exp [-(1-o(1))\sum_{i=t_0}^{t-1} \frac{ik^4}{2n^2}]
=\exp [-(1-o(1))\C{t}{2}\frac{k^4}{2n^2}]
$,}
\]
completing the proof of Theorem~\ref{TMP2}(b).\end{proof}

\begin{proof}[Proof of Claim 2.]
Let $\sss=\{P_1\dots P_\mm\}$ be the set of pairs contained in (at least one of)
$A_1\dots A_{i-1}$ and not meeting $W_{i-1}$,
and
\[
\T= \{X\sub [n]: \Cc{X}{2}\cap \sss=\0\}\supseteq
\{X\sub [n]: |X\cap A_j|\leq 1 ~\forall j<i\}.
\]
It is enough to bound $\pr(A_i\in \T)$.
In view of $\pee_{i-1}$, the number of pairs covered by $A_1\dots A_{i-1}$ is
$(i-1)\C{k}{2}$, while $\Q_{i-1}$ says that the number of these that meet $W_{i-1}$ is at most
$\gd t_0k(k-1)$; thus $m$ ($=|\sss|$) $\sim i k^2/2$.
Note also that the number of non-disjoint (unordered) pairs from $\sss$ is less than
\[
\mbox{$\sum d_\sss^2(j)/2 ~< ~a_0k\sum d_\sss (j)/2\sim \gd n ik/2$.}
\]

For a silly technical reason (see \eqref{Jlast}) we now treat $t\ll n^2/k^4$ separately.
Let $R_l=\{A_i\supseteq P_l\}$ ($l\in [\mm]$).  Then using $\pr(A_i\supseteq I)\sim (k/n)^{|I|}$
for fixed $|I|$ together with the above asymptotics yields
\[
\mbox{$\pr (A_i\not\in \T)\geq \sum\pr(R_l) -\sum \sum \pr(R_lR_{l'})\sim ik^4/(2n^2)$.}
\]
(The first sum is asymptotic to the r.h.s.\  and the double sum is asymptotically
at most $(\gd nik/2)(k/n)^3 + (ik^2/2)^2(k/n)^4 = \gd ik^4/(2n^2) + i^2k^8/(4n^4) $,
which is $o(ik^4/n^2)$ since we assume $t\ll n^2/k^4$.)

\bn

Now assume $t=\gO(n^2/k^4)$ (so $\eps =\gO(1/k)$).
Set $p=(1-\gd) k/n$
and let $B$ be the random subset of $[n]$ gotten by including
each element with probability $p$, independent of other choices.
According to the ``Basic Janson Inequality"
(\cite{JLRJanson} or e.g.\ \cite[Ch. 8]{AS})),
\beq{JI}
\pr(B\in \T)\leq e^{-\mu+\gD},
\enq
where (\emph{cf.} the above discussion for $t\ll n^2/k^4$)
$\mu=\mm p^2\sim ik^4/(2n^2)$ and
\[    
\gD ~=~
\mbox{$\frac{1}{2}\sum_j d_\sss(j)(d_\sss(j)-1) p^3
~<~(1-o(1))\gd ik^4/(2n^2)~\ll~\mu$}.
\]
Thus \eqref{JI} gives the desired bound with $B $ in place of $A_i$; that is,
\[
\pr(B\in \T) <
\exp [-(1-o(1))ik^4/(2n^2)].
\]
Finally, we combine this with
$\pr( |B|> k) < \exp[-\gO(\gd^2 k)]$ (see Theorem~\ref{thm:Chernoff}) to obtain
\begin{eqnarray}
\pr(A_i\in \T) &\leq & \pr(B\in \T||B|\leq k)
< \pr(B\in \T)/\pr( |B|\leq k)\nonumber\\
&=& \pr(B\in \T)(1+e^{-\gO(\gd^2 k)})
<\exp [-(1-o(1))ik^4/(2n^2)]~~~~~\label{Jlast}
\end{eqnarray}
(note $\eps=\gO(1/k)$ and the assumed $\exp[-\gO(\gd^2k)]\ll \gd$
give
$\exp[-\gO(\gd^2k)]\ll \eps \gd k = t_0k^4/n^2$).
\end{proof}


\section{Proof of Theorem~\ref{TMP}}\label{SecMP}

We will give two proofs; the second is easier and proves more
(\emph{as far as we can see,} the first handles only the more interesting case of larger $c$), but
we include the first, which was our original argument, as it seems to us
the more interesting.
We will not try to optimize the implied constants in \eqref{ProbM}.

In each proof the following observation,
which is where we use near-disjointness,
will play a key role.
For an $l$-graph $\g$ and set $e$ (in practice a member of $\g$), let
\[
I(e,{\cal G}) =|\{ g\in {\cal G}: e\cap g \not= \emptyset\}|.
\]
\begin{prop}\label{CSProp}
For any nearly-disjoint $l$-graph $\g$ on n vertices and $\gd>0$,
\beq{B large}
|\{e\in \g:I(e,\g) < \gd |\g|l^2/n\}| <\delta |\g| + n/\llll .
\enq
\end{prop}

\begin{proof}
Writing $\sss$ for the set in \eqref{B large}, we have
$
\sum_{x}d_{\sss}(x)= \llll  |\sss|.
$
On the other hand, near-disjointness implies that
for any $e\in \sss$,
\[
\mbox{$\sum_{x\in e}d_\g(x)
< \delta |\g|\llll ^2/n +\llll -1,$}
\]
yielding
\[
\mbox{$\sum_{x}d_{\sss}(x)d_\g(x)
=\sum_{e\in \sss}\sum_{x\in e}d_\g(x)
<|\sss|(\delta |\g|\llll ^2/n +\llll -1)$}
\]
and
\[
\mbox{$n\sum_x d^2_{\sss}(x)
~\leq ~n \sum_{x}d_{\sss}(x)d_{{\g}}(x)
~< ~ |\sss|(\delta |\g|\llll ^2+n(\llll -1))$.}
\]
Combining and using Cauchy-Schwarz we have
\[
\mbox{$\llll ^2|\sss|^2 =
\left( \sum_{x}d_{\sss}(x)\right)^2
< |\sss|(\delta |\g|\llll ^2+n(\llll -1)),$}
\]
which implies \eqref{B large}.
\end{proof}

\begin{proof}[First proof of Theorem~\ref{TMP}.]
Here $\m $ and $\sss$ will always be
matchings (of $\h$) of sizes $m$ and $\gc m$ respectively.
As indicated above, we are now considering only the second regime in \eqref{ProbM},
so may assume $c$ is a bit large.
To bound the number of $\m$'s we first want an $\sss\sub \m$
for which the number of possible continuations $\m\sm\sss$ is ``small."
(The parameters $\gc,\gd,\vt$ will be set below.)

Given $\sss$, set
\[
\R =\R_\sss = \{e\in {\cal H}: I(e, {\cal S})=0\},
\]
\[   
\B=\B_\sss =\{e\in {\cal R}: I(e,{\cal R})\ge \delta |\R|\llll ^2/n\}
\]   
and $|\R|=r$.
If $\m \supseteq\sss$ then, trivially,
\beq{M-S in R}
{\cal M}\setminus {\cal S} \sub \cal R;
\enq
so the number of $\m$'s containing $\sss$ is at most
\beq{firstbd}
\Cc{r}{(1-\gamma)m}.
\enq
Note also that Proposition~\ref{CSProp} gives
\beq{B large'}
|{\cal R}\setminus {\cal B}| <\delta r + n/\llll ~=:r^*.
\enq
Since we will choose $\gd$ fairly small,
\eqref{B large'} (with \eqref{M-S in R}) will limit possibilities for
$(\m\sm\sss)\sm \B$.
We
next show that for any $\m $ there is some $\sss$
for which $(\m\sm \sss)\cap \B$ is small.

Given $\m $ and $\sss\sub \m$, with $\R,\B,r$ as above, let
\[
{\cal R}_1 = \{e\in {\cal R}: I(e,{\cal M})< \vartheta m\llll ^2/n\},
 ~~~{\cal R}_2={\cal R}\setminus{\cal R}_1,
\]
\[
\A_1=\{e\in {\cal M}\cap {\cal B}: I(e,{\cal R}_1) \ge \delta r\llll ^2/(2n)\}
\]
and $\A_2=(\m\cap \B)\sm \A_1$; thus
$e\in \A_2 $ implies $ I(e,{\cal R}_2) > \delta r\llll ^2/(2n).$

We then want to bound $|\A_1|$ and $|\A_2|$, the first in general, the second for a suitable $\sss$.
In each case we consider
\[
N_i =|\{(e,f)\in \A_i\times \R_i: e\cap f\neq \0\}|.
\]
For $i=1$, we have
$
|\A _1|\delta r\llll ^2/(2n)\leq N_1< |\R_1|\vartheta m \llll ^2/n,
$
implying
\beq{A_1 small}
|\A _1| \leq 2m\vartheta/\delta.
\enq
For $i=2$ we again have $N_2\geq |\A _2|\delta r\llll ^2/(2n)$,
but our upper bound now depends
on $\sss$.
Suppose $\sss$ is chosen uniformly from $\C{\m}{\gc m}$ (so
$\cal R$, $\cal B$, ${\cal R}_1$, and ${\cal R}_2$ are also random).
Set
$
\cee =\{f\in {\cal H} : I(f,{\cal M})\ge  \vartheta m\llll ^2/n\}
$
(the edges that will be in $\R_2$ if they are in $\R$).
Then
\beq{gG2}
\mbox{$\E N_2 ~< ~\sum_{f\in \cee}I(f,{\cal M})(1-\gamma)^{I(f,{\cal M})}
~<~
\sum_{f\in \cee} I(f,{\cal M})e^{-\gamma I(f,{\cal M})}.$}
\enq
Since $xe^{-\gamma x}$ is decreasing for $x\ge  1/\gc$
and---a very small point---we will choose parameters so
\beq{detail}
\vartheta m\llll ^2/n\geq 1/\gc,
\enq
the r.h.s.
of \eqref{gG2} is at most
$
t (\vt m\llll ^2/n)e^{-\gamma  \vartheta m\llll ^2/n}
$.
Thus each $\m $ admits some $\sss$
for which $N_2$ is at most this value, implying
\[   
|\A _2|\leq \tfrac{2\vartheta}{\delta}\, \tfrac{t}{r} e^{-\gamma  \vartheta m\llll ^2/n}m
\]   
and (recall \eqref{A_1 small})
\beq{mBbd}
|\m\cap\B| ~\le ~\tfrac{2\vt}{\gd}
(1+ \tfrac{t}{r}e^{-\gamma  \vartheta m\llll ^2/n})m ~=:~ s_r.
\enq
Thus the number of choices for $\m $ is at most
the number of ways to
choose $\sss$ and then an $\m\supseteq\sss$ satisfying \eqref{mBbd}.

\mn
\emph{Remark.}
Note we are not \emph{choosing} the $\R_i$'s and $\A_i$'s (which do depend on $\m $);
these are just used in establishing existence of the desired $\sss$.

\mn

For a given $\sss$ the number of choices for $\m\supseteq \sss$
satisfying \eqref{mBbd} is at most
\beq{psi}
\mbox{$\psi = \max_r
\min\left\{\Cc{r}{(1-\gamma)m},\
\sum_{\aaa\leq s_r}\C{t-\gc m}{\aaa}{r^* \choose m-\gamma m-a}\right\}$,}
\enq
where $r=|\R_\sss|$
(see \eqref{B large'} for $r^*$)
and the first bound is from \eqref{firstbd}.
We may thus bound
the number of $\m $'s by $\C{t}{\gc m}\psi$, and
the probability in \eqref{ProbM} by
\beq{finalprob}
\Cc{t}{ m}^{-1}\Cc{t}{\gc m}\psi.
\enq

\mn

Finally, we need to set parameters and discuss bounds.
Set $\gc=0.1$, $\gd = 100c^{-1}\log c$ and $\vt =0.1\gd$.
(Note these support \eqref{detail}.)
For $r<\gd t$ we use the first bound in \eqref{psi} to say the expression
in \eqref{finalprob} is at most
\[    
\Cc{t}{ m}^{-1}\Cc{t}{\gc m}\Cc{\gd t}{(1-\gc)m}
<
\Cc{m}{\gc m}\gd^{(1-\gc)m} ~<~ \exp[-((1-\gc)\log (1/\gd)-1)m].
\]    

For $r\geq \gd t$, referring to \eqref{mBbd}, we have $s:=s_r < 0.2[1+ 1/(\gd c)]m < 0.3m$.
So, using the second part of \eqref{psi}, we may bound the expression in \eqref{finalprob}
by
\[
s\Cc{t}{ m}^{-1}\Cc{t}{\gc m}
\Cc{t-\gc m}{s}\Cc{r^*}{m-\gamma m-s}
<\exp [- 0.5\log(1/\gd) m],
\]
where we used
\[
\Cc{t}{ m}^{-1}\Cc{t}{\gc m}
\Cc{t-\gc m}{s} =\Cc{m}{\gc m,s,m-\gc m -s}\Cc{t-\gc m -s}{m-\gc m -s}^{-1}
\]
and, say, $r^* < 2\gd(t-\gc m-s)$.
\end{proof}

\begin{proof}[Second proof of Theorem~\ref{TMP}.]

Here
it will be easier to consider
$e_1\dots e_m$ drawn
uniformly and \emph{independently} from $\h$ and prove bounds as in \eqref{ProbM}
for the probability that these $e_i$'s form a matching.
This is equivalent since
\[
\gz = \pr(\mbox{the $e_i$'s form a matching})/\pr(\mbox{the $e_i$'s are distinct})
\]
(recall $\gz =\pr(\mbox{$\m$ is a matching})$),
and the denominator (roughly $\exp[-\frac{m^2}{2t}]$) doesn't significantly
affect the bounds in \eqref{ProbM}.

\mn

Let $\h_0=\h$ and, for $j\geq 1$,
\[
\h_j=\{e\in \h:e\cap e_i=\0 ~ \forall i\leq j\}.
\]
(Thus $\h_0\supseteq \h_1\supseteq \cdots $ and the $e_i$'s form a matching iff
$e_i\in \h_{i-1}$ for all $\forall i\in [m]$.)
Set
\[
\gd =\left\{\begin{array}{ll}
e^{-1}&\mbox{if $c\leq e$,}\\
c^{-1}\log c &\mbox{otherwise}
\end{array}\right.
\]
(the precise values are not very important) and
\[
\cee_j=\{e\in \h_j: I(e,\h_j)<\gd |\h_j|l^2/n\},
\]
and let $\Q $ be the event
\[
\{|\{j\in [m/2]:e_j\in \cee_{j-1}\}|> m/3\}.
\]


\mn
Proposition~\ref{CSProp} gives
\beq{Csmall}
|\cee_j| < \gd |\h_j| + n/l,
\enq
so that
we always (regardless of history) have
\[
\pr(e_j\in \cee_{j-1})< \gd+n/(lt)=:\gd' <
\left\{\begin{array}{ll}
e^{-1}+o(1)&\mbox{if $c\leq e$,}\\
c^{-1}\log c +\min\{c^{-1},o(1)\}&\mbox{otherwise,}
\end{array}\right.
\]
and in either case $\gd'<1/2$.
Thus
$|\{j\in [m/2]:e_j\in \cee_{j-1}\}|$
is stochastically
dominated by a r.v.\ with the distribution ${\rm Bin}(m/2,\gd')$, and Theorem~\ref{thm:Chernoff} gives
\beq{Qbd}
\pr (\Q) < \exp [-\gO(\log (1/\gd'))m] .
\enq

\mn
(Of course $\gO(\log (1/\gd'))$ is just $\gO(1)$ until $c$ is a bit large.)
On the other hand, if $\Q$ does not occur then
\[
|\h_{m/2}| < (1-\gd l^2/n)^{m/6}t ,
\]
so the probability that we continue to a matching is less than
\beq{Cbd}
(1-\gd l^2/n)^{(m/6)\cdot (m/2)} <
\exp [-\gd cm/(12)].
\enq
The theorem follows.
\end{proof}

\section{Lower bound}\label{LB}

As noted earlier, Proposition~\ref{Plb} is an application of the following
celebrated result of Ajtai, Koml\'os and Szemer\'edi \cite{aks,sidon}.
\begin{thm}\label{TAKS}
There is a fixed $c>0$  such that $\ga(\gG) >c (N \log D)/D$
for any triangle-free graph $\gG$ with N vertices
and average degree at most D.
\end{thm}
\nin
(As usual $\ga$ is independence number).
The reduction to Theorem~\ref{TAKS} is quite routine and we will not
give the full blow-by-blow.

\mn

In what follows we set $M=f(k) $ ($=\C{n}{k}2^{-\C{k}{2}}$).
We want existence of a large independent set in the graph $\gG$ whose vertices are the
$k$-cliques of $G$ ($=G_{n,1/2}$) and edges the pairs that share edges (of $G$).
A standard application of the second moment method (e.g.\ \cite[Sec.\ 4.5 and Cor.\ 4.3.5]{AS}
gives $|V(\gG)| > (1-o(1))M$ w.h.p.\ (meaning, as usual, with probability tending to 1
as $n\ra\infty$), and routine analysis shows that the expected numbers of edges and triangles
in $\gG$ are respectively at most $k^4M^2/(2n^2)$ and $2k^6M^3/(3n^4)$.
(The main contributions are from edges consisting of pairs of cliques with just one common edge
and triangles composed of three cliques sharing the same edge.)
So Markov's Inequality says that with probability at least $1/6-o(1)$, we have
\beq{Gamma}
|V(\gG)| \sim M, ~~~ |E(\gG)| < k^4M^2/n^2 ~~\text{and} ~~ |T(\gG)|< 2k^6M^3/n^4,
\enq
where $T$ denotes number of triangles.  Thus Proposition~\ref{Plb}
will follow from the next assertion.

\mn
\emph{Claim.}
If \eqref{Gamma} holds then $\ga(\gG) > \gO(k^{-4}n^2 \log k)$.

\mn
To see this let $\gd = n^2/(2k^3M)$ and consider the subgraph $H$ of $\gG$ induced by
$W$ chosen uniformly from the subsets of $V(\gG)$ of size $\gd M \sim n^2/(2k^3)$.
Then $\E|E(H)| < \gd^2|E(\gG)|$ and $\E|T(H)| < \gd^3|T(\gG)|$, so (again using Markov)
there is a choice of $W$ for which
$|E(H)| \leq n^2/k^2$ and $|T(H)| \leq n^2/(3k^3) $.
We may then find some triangle-free $K\sub H$ on $(1/3-o(1))n^2/k^3$ vertices with average
degree at most $\frac{2n^2/k^2}{|V(K)|} =O(k)$, and applying Theorem~\ref{TAKS} gives the claim.\qed

\end{document}